\documentclass{article}
\usepackage{amsfonts}
\usepackage{amsmath}

\newtheorem{theorem}{Theorem}

\newtheorem{conclusion}[theorem]{Conclusion}

\newtheorem{corollary}[theorem]{Corollary}

\newtheorem{lemma}[theorem]{Lemma}

\newtheorem{proposition}[theorem]{Proposition}
\newtheorem{remark}[theorem]{Remark}

\newenvironment{proof}[1][Proof]{\noindent\textbf{#1.} }{\ \rule{0.5em}{0.5em}}
\input{tcilatex}

\begin{document}

\title{A surprising regularizing effect of the nonlinear semigroup
associated to the semilinear heat equation and applications to reaction
diffusion systems}
\author{Said Kouachi \\
University of Abbes Laghrour Khenchela. Algeria.\\
E-mail: kouachi@univ-khenchela.dz}
\maketitle

\begin{abstract}
It is well known to prove global existence for the semilinear heat equation
via the well known regularizing effect, we need to show that the reaction is
uniformly bounded in the Lebesgue space $L^{\infty }(\left( 0,T_{\max
}\right) ,L^{p}(\Omega ))$ for some $p>n/2$ where $\left( 0,T_{\max }\right) 
$ and $\Omega $ \ are respectively the temporal and spatial spaces and $%
\Omega $ \ is an open bounded domain of $\mathbb{R}^{n}$. In this paper we
prove that if the reaction (even it depends on the temporal and the spatial
variables) preserves the same sign after some time $t_{0}\in \left(
0,T_{\max }\right) $, then the solution is global provided it belongs to the
space $L^{\infty }(\left( t_{0},T_{\max }\right) ,L^{1}(\Omega ))$. That is
if the sign of the reaction is preserved, then positive weak solutions for
quasilinear parabolic equations on bounded domains subject to homogenous
Neumann boundary conditions become classical and exist globally in time
independently on the nonlinearities growth. We apply this result to coupled
reaction diffusion systems and prove that weak solutions become classical
and global provided that the reactions become of constant sign after some
time $t_{0}\in \left( 0,T_{\max }\right) $ and belong to the space $%
L^{\infty }(\left( t_{0},T_{\max }\right) ,L^{1}(\Omega ))$. The
nonlinearities growth isn't taken in consideration. The proof is based on
the maximum principle.
\end{abstract}

We consider a semilinear Heat equation and corresponding reaction diffusion
systems. We give conditions which guarantee global existence of solutions
with positive initial data. The problems of global existence (and obviously
blow-up at finite time) of solutions for nonlinear parabolic equations and
systems with Neumann boundary conditions have been investigated extensively
by many authors (see e.g., \cite{Den-Lev}, \cite{Gal-Vaz}, \cite{Lev}, \cite%
{Zha}, \cite{Zha2006}, \cite{Qi}, \cite{Qi95}, \cite{Sam-Gal-Kur-Mik}, \ 
\cite{Den}, \cite{Din-Li}, \cite{Sou-Jaz-Mon}, \cite{Ish-Miz}, \cite{Ish-Yag}%
, \cite{Jaz-Kiw} and \cite{Miz} for equations and \cite{Ama 85}, \cite%
{Pie-Sch 22}, \cite{Kou(MMAS)}, \cite{Kou(DPDE)} and \cite{Sou 18} for
systems) and the references therein. We apply the results obtained for the
semilinear Heat equation to some reaction\ diffusion systems. We begin with
the following semi linear heat equation:

\begin{equation}
\left\{ 
\begin{array}{l}
\frac{\partial u}{\partial t}-a\Delta u=f(t,x,u)\text{\qquad \qquad in }%
\mathbb{R}^{+}\times \Omega , \\ 
\frac{\partial u}{\partial \eta }=0\text{\qquad \qquad on }\mathbb{R}%
^{+}\times \partial \Omega , \\ 
u(0,x)=u_{0}(x)>0,\qquad \text{in}\;\Omega ,%
\end{array}%
\right.  \label{Heat 1.1}
\end{equation}%
where $\dfrac{\partial }{\partial \eta }$ denotes the outward normal
derivative on the boundary $\partial \Omega $ of an open bounded domain $%
\Omega \subset \mathbb{R}^{n}$ of class $\mathbb{C}^{1}$ and $a$ is a
positive constant. Assume that the reaction $f$\ \ is continuously
differentiable and we have for some $t_{0}\in \left( 0,T_{\max }\right) $
the following%
\begin{equation}
f(t,x,u)\neq 0,\ \text{for all }t\in \left( t_{0},T_{\max }\right) ,\ x\in
\Omega ,\ u>0\text{.}  \label{Heat 1.7}
\end{equation}%
Assume that%
\begin{equation}
\ f(t,x,0)\geq 0,\ \text{for all}\ t\geq 0,\ x\in \Omega ,  \label{P}
\end{equation}%
which assures the positivity of the solution on$\ \left( 0,T_{\max }\right)
\times \Omega $ for all nonnegative initial data via standard comparison
arguments for parabolic equations (see e.g. D. Henry \cite{Hen} and its
references). It is known from the classical parabolic equation theory (see
e.g. \cite{Ama 85} and its references) that there exists a unique local
classical solution $u$ to problem\ (\ref{Heat 1.1}) defined on the interval $%
\left( 0,T_{\max }\right) $ where $T_{\max }$ denotes the eventual
blowing-up time in $\mathbb{L}^{\infty }(\Omega )$. Some forms of (\ref{Heat
1.1}) have been treated already. Lair and Oxley \cite{Lai-Oxl} considered
the following problem:

\begin{equation}
\left\{ 
\begin{array}{l}
\frac{\partial u}{\partial t}-\nabla .\left( a\left( u\right) \nabla
u\right) =f(u)\text{\qquad \qquad in }\mathbb{R}^{+}\times \Omega , \\ 
\frac{\partial u}{\partial \eta }=0\text{\qquad \qquad on }\mathbb{R}%
^{+}\times \partial \Omega , \\ 
u(0,x)=u_{0}(x)>0,\qquad \text{on}\;\Omega .%
\end{array}%
\right.  \label{Heat 1.2}
\end{equation}%
where the functions $a$ and $f$ \ are assumed to be nondecreasing,
nonnegative and satisfying%
\begin{equation*}
a\left( s\right) f(s)>0,\text{ for all }s>0,
\end{equation*}%
with $a(0)=0$. They proved the existence of global and blow-up of
generalized solutions (solutions that are limit sequence of certain
approximating problems). More precisely, they proved blow up at finite time
(and obviously global existence) if and only if%
\begin{equation}
\underset{0}{\overset{\infty }{\int }}\frac{ds}{1+f\left( s\right) }<+\infty
,  \label{b.u.}
\end{equation}%
which is equivalent (in the case of the blow up) to%
\begin{equation*}
\underset{\alpha }{\overset{\infty }{\int }}\frac{ds}{f\left( s\right) }%
<+\infty ,
\end{equation*}%
for some positive constant $\alpha $. A similar phenomenon occurs for the
homogenous Dirichlet problem (see \cite{Gal}, \cite{Lev} and their
references). When the reaction doesn't depend on $\left( t,x\right) $, using
comparison principles with the ODE associated with problem (\ref{Heat 1.1})
we can show easily global existence for $p<1$ and blow-up at finite time for 
$p>1$ for appropriate initial conditions. For homogeneous Robin boundary
conditions, the authors in \cite{Beb}, \cite{Bel} and \cite{Lac} considered
the following problem%
\begin{equation}
\left\{ 
\begin{array}{l}
\frac{\partial u}{\partial t}-a\Delta u=\lambda f(u)\text{\qquad \qquad in }%
\mathbb{R}^{+}\times \Omega , \\ 
\alpha \frac{\partial u}{\partial \eta }+\beta u=0\text{\qquad \qquad on }%
\mathbb{R}^{+}\times \partial \Omega , \\ 
u(0,x)=u_{0}(x)>0,\qquad \text{on}\;\Omega ,%
\end{array}%
\right.  \label{1.2.1}
\end{equation}%
where $\alpha $ and $\beta $ are nonnegative functions and the function $f$
\ is required to satisfy other restrictions as the convexity, strict
positivity and (\ref{b.u.}). They established finite time blowup provided
the positive parameter $\lambda $ is greater than some critical value. For
homogeneous Neumann boundary conditions, as we study here, there seem to be
no many authors (see \cite{Bel}, \cite{Ima-Moc} and their references) worked
on the problem at hand. Bellout \cite{Bel} proved blowup under the
restrictive conditions on $f$ as cited previously. Imai and Mochizuki \cite%
{Ima-Moc} studied the following problem:%
\begin{equation}
\left\{ 
\begin{array}{l}
\frac{\partial \left( h\left( u\right) \right) }{\partial t}-a\Delta u=f(u)%
\text{\qquad \qquad in }\mathbb{R}^{+}\times \Omega , \\ 
\frac{\partial u}{\partial \eta }=0\text{\qquad \qquad on }\mathbb{R}%
^{+}\times \partial \Omega , \\ 
u(0,x)=u_{0}(x)>0,\qquad \text{on}\;\Omega .%
\end{array}%
\right.  \label{Heat 1.3}
\end{equation}%
In addition to the multiple constraints on $a$ and $f$, one of which relates
to (\ref{b.u.}), the adequate conditions for the existence of global and
blow-up solutions additionally request that the nonnegative initial
condition to be sufficiently large. Gao, Ding and Guo \cite{Gao-Din-Guo}
treated the following problem:%
\begin{equation}
\left\{ 
\begin{array}{l}
\frac{\partial \left( h\left( u\right) \right) }{\partial t}-\nabla .\left(
a\left( u\right) \nabla u\right) =f(u)\text{\qquad \qquad in }\mathbb{R}%
^{+}\times \Omega , \\ 
\frac{\partial u}{\partial \eta }=0\text{\qquad \qquad on }\mathbb{R}%
^{+}\times \partial \Omega , \\ 
u(0,x)=u_{0}(x)>0,\qquad \text{on}\;\Omega ,%
\end{array}%
\right.  \label{Heat 1.4}
\end{equation}%
and obtained sufficient conditions for the existence of global solution and
their blow-up. Meanwhile, the upper estimate of the global solution, the
upper bound of the \textquotedblleft blow up time\textquotedblright\ and
upper estimate of \textquotedblleft blow-up rate\textquotedblright\ were
also given. The authors in \cite{Din-Guo} generalized the results of \cite%
{Gao-Din-Guo} to the following problem:%
\begin{equation}
\left\{ 
\begin{array}{l}
\frac{\partial \left( h\left( u\right) \right) }{\partial t}-\nabla .\left(
a\left( t,u\right) b\left( x\right) \nabla u\right) =g\left( t\right) f(u)%
\text{\qquad \qquad in }\mathbb{R}^{+}\times \Omega , \\ 
\dfrac{\partial u}{\partial \eta }=0\text{\qquad \qquad on }\mathbb{R}%
^{+}\times \partial \Omega , \\ 
u(0,x)=u_{0}(x)>0,\qquad \text{on}\;\Omega ,%
\end{array}%
\right.  \label{Heat 1.5}
\end{equation}%
by constructing auxiliary functions and using maximum principles and
comparison principles under some appropriate assumptions on the functions $a$%
, $b$, $f$ , $g$, and $h$.

\begin{remark}
It is easy to note that in the above examples the reactions are independent
on the temporal and the spatial variables where here (i.e. system (\ref{Heat
1.1})) they are taken in consideration.
\end{remark}

The norms in the spaces $L^{\infty }(\Omega )\ $(or $C\left( \overline{%
\Omega }\right) $) and $L^{p}(\Omega )$, $1\leq p<+\infty $ are denoted
respectively by%
\begin{equation*}
\left\Vert u\right\Vert _{\infty }\;=\;\max_{x\in \Omega }\left\vert
u(x)\right\vert ,
\end{equation*}%
and%
\begin{equation*}
\left\Vert u\right\Vert _{p}=\left( \dint\limits_{\Omega }\left\vert
u(x)\right\vert ^{p}dx\right) ^{\frac{1}{p}}.
\end{equation*}

Let us recall that global existence of solutions of problem (\ref{Heat 1.1})
is obtained by application to the reaction (for $q=+\infty $ and $p>n/2$)
the following $L^{p}-L^{q}$ property (called regularizing or smoothing
effect) of the semigroup $S\left( t\right) $ associated to the operator $%
\Delta $ in $L^{\infty }(\Omega ).$

\begin{proposition}
\label{regularizing effect}(\textbf{regularizing effect}) If a \textit{%
semigroup }$\left\{ S(t)\right\} _{t\geq 0}$\textit{\ is strongly continuous}%
, then for all $1\leq p<q\leq \infty $, there exists a positive constant $C$
such that%
\begin{equation}
\left\Vert S(t)v\right\Vert _{q}\leq Ct^{-\frac{n}{2}\left( \frac{1}{p}-%
\frac{1}{q}\right) }\left\Vert v\right\Vert _{p},\text{ for all }v\in 
\mathbb{L}^{p}\left( \Omega \right) .  \label{Reg. Eff.}
\end{equation}
\end{proposition}

Let us denote \bigskip $\Omega _{T}=\left( 0,T\right) \times \Omega $ and $%
\Gamma _{T}=\left\{ 0\right\} \times \Omega \cup \left[ 0,T\right] \times
\partial \Omega ,$ the following theorem is the parabolic weak simplified
maximum principle for the heat equation.

\begin{proposition}
Suppose $x\rightarrow u(t,x)$ is in $C^{2}\left( \Omega _{T}\right) \cap
C\left( \overline{\Omega }\right) $ and $t\rightarrow u(t,x)$ is $%
C^{1}\left( \left[ 0,T\right] \right) $. Suppose that%
\begin{equation}
\dfrac{\partial u}{\partial t}\leq a\Delta u,\ \ \text{on }\Omega _{T},
\label{Heat 1.6}
\end{equation}%
then%
\begin{equation*}
\max \left\{ u(t,x):\ (t,x)\in \overline{\Omega _{T}}\right\} =\max \left\{
u(t,x):\ (t,x)\in \Gamma _{T}\right\} .
\end{equation*}%
If the inequality (\ref{Heat 1.6})\bigskip\ is replaced by the following%
\begin{equation}
\dfrac{\partial u}{\partial t}\geq a\Delta u,\ \ \text{on }\Omega _{T}
\label{Heat 1.6-}
\end{equation}%
then%
\begin{equation}
\min \left\{ u(t,x):\ (t,x)\in \overline{\Omega _{T}}\right\} =\min \left\{
u(t,x):\ (t,x)\in \Gamma _{T}\right\} .  \label{Heat 1.7-}
\end{equation}
\end{proposition}

More general forms of the maximum principle\ can be found in \cite{Pro-Wei
67}, \cite{Pro-Wei 84}, \cite{Kus}, \cite{Nir}, \cite{Kus-ellip} and their
references.

\section{\textbf{Statement and proof of the main results}}

\bigskip Our results are based on the following Lemma

\begin{lemma}
Suppose that a positive solution $u$ of (\ref{Heat 1.1}) attains the same
value more than two times on $\Omega _{T}$, then the reaction vanishes at
last one time on $\Omega _{T}$.
\end{lemma}

The above lemma is with great interest in the subsequent and means if $u$
solves (\ref{Heat 1.1}) and satisfies $\ $ 
\begin{equation}
u(t_{1},x_{1})=u(t_{2},x_{2})=u(t_{3},x_{3}),  \label{Heat 1.2.3}
\end{equation}%
for some distinct $\left( t_{i},x_{i}\right) \in \Omega _{T},\ i=1,2,3$ with 
$t_{1}<t_{2}<t_{3}.$\ Then there exists $t\in \left( t_{1},t_{3}\right) $
and $x\in \Omega $ such that 
\begin{equation}
f(t,x,u(t,x))=0.  \label{Heat 0}
\end{equation}

\begin{proof}
Let us denote $u\left( t_{i},x_{i}\right) =c,\ i=1,2,3$, where $c$ is a
positive constant. The function $w=u-c$ will possess three successive zeros
unless it is a constant function. The regularity proprieties of $u$ will
imply that it possess two extremums: If, for example between $\left(
t_{1},x_{1}\right) $ and $\left( t_{2},x_{2}\right) $ the function $u$ has a
local maximum where $\frac{\partial u}{\partial t}=0\geq \Delta u$ at which $%
f$ \ is non positive. Then automatically between $\left( t_{2},x_{2}\right) $
and $\left( t_{3},x_{3}\right) $ it will possess a local minimum where $%
\frac{\partial u}{\partial t}=0\leq \Delta u$ for which $f$ \ is
nonnegative. The local minimum isn't not zero via the maximum principle
which assures that $u$\ can't attain its minimum in $\Omega _{T}$. Using the
equation (\ref{Heat 1.1}), we deduce that the reaction $f(t,x,u)$ will
possess at last a zero on $\Omega _{T}$.
\end{proof}

Our main result concerning the heat equation is the following

\begin{theorem}
\label{Max}\bigskip Assume that (\ref{P}) and (\ref{Heat 1.7}) are
satisfied, then solutions of (\ref{Heat 1.1}) belonging to $L^{\infty
}(\left( t_{0},T_{\max }\right) ,L^{1}(\Omega ))$ are classical, global and
uniformly bounded in time.
\end{theorem}

\begin{proof}
Put 
\begin{equation}
M_{1}=\underset{t_{0}\leq t<T_{\max }}{\max }\left\Vert u\left( t,.\right)
\right\Vert _{1},  \label{Heat 1.9}
\end{equation}%
and let $\epsilon $ a positive constant satisfying%
\begin{equation}
\epsilon >\max \left\{ 1,\ \frac{M_{1}}{\alpha \left\vert \Omega \right\vert 
}\right\} ,  \label{Heat 1.12}
\end{equation}%
where%
\begin{equation*}
\alpha =\left\Vert u\left( t_{0},.\right) \right\Vert _{\infty }.
\end{equation*}%
Let $\overline{t}$ the greatest $t\in \left( t_{0},T_{\max }\right) $ such
that%
\begin{equation*}
u-\epsilon \alpha =0,\ \text{for some }x\in \Omega \text{,}
\end{equation*}%
which means that%
\begin{equation}
u-\epsilon \alpha \neq 0,\ \text{for all }x\in \Omega \text{ and all }t\in
\left( \overline{t},T_{\max }\right) .  \label{Heat 1.8}
\end{equation}%
If such $\overline{t}$ doesn't exist, then we shall have the following
alternative:%
\begin{equation*}
u-\epsilon \alpha \neq 0,\ \text{for all }x\in \Omega \text{ and all }t\in
\left( t_{0},T_{\max }\right) ,
\end{equation*}%
or for all $\overline{t}\in \left( t_{0},T_{\max }\right) $, there will
exist $t>\overline{t}$ and $x\in \Omega $ such that $u-\epsilon \alpha =0$.
For the first alternative, since $u\left( t_{0},x\right) <\epsilon \alpha $,
then we have%
\begin{equation*}
u<\epsilon \alpha ,\ \text{for all }x\in \Omega \text{ and all }t\in \left(
t_{0},T_{\max }\right) ,
\end{equation*}%
and $u$ is uniformly bounded on $(t_{0},T_{\max }[\times \Omega $. Using the
continuity of $u$ on the remaining set $\left[ 0,t_{0}\right] \times \Omega $
with the above inequality we can say that $u$ is uniformly bounded on $%
\left( 0,T_{\max }\right) \times \Omega .$ For the second alternative, the
function $w=u-\epsilon \alpha $ will possess an infinity of zeros unless it
is a constant function. The regularity proprieties of $u$ will imply that it
has an infinity of extremums. If we denote, for example $\left(
t_{k},x_{k}\right) $, $k=1,\ 2,\ 3$ three successive zeros of the function $%
w $ in $\left( t_{0},T_{\max }\right) \times \Omega $, then the solution $u$
of (\ref{Heat 1.1}) will attain the same value more than two times on $%
\Omega _{T}$. The above lemma applied for $c=\epsilon \alpha $ will
contradict the condition (\ref{Heat 1.7}). Consequently we have (\ref{Heat
1.8}) and another time the following alternative is presented: \ \ \ \ \ \ \
\ \ \ \ \ \ \ \ \ \ \ \ \ \ The first one is%
\begin{equation}
u-\epsilon \alpha >0,\ \text{for all }x\in \Omega \text{ and all }t\in
\left( \overline{t},T_{\max }\right) ,  \label{nu14}
\end{equation}%
and the second is%
\begin{equation}
u\leq \epsilon \alpha ,\ \text{for all }x\in \Omega \text{ and all }t\in
\left( \overline{t},T_{\max }\right) .  \label{Heat 1.13}
\end{equation}%
The first inequality will give%
\begin{equation}
M_{1}\geq \epsilon \alpha \left( \int\limits_{\Omega }1dx\right) =:\epsilon
\alpha \left\vert \Omega \right\vert ,\ \ \ \text{\ on }\left( \overline{t}%
,T_{\max }\right) ,  \label{Heat 1.10}
\end{equation}%
which contradict the inequality (\ref{Heat 1.12}). We conclude that the
second alternative (i.e. \ref{Heat 1.13}) is always satisfied\ on $\left( 
\overline{t},T_{\max }\right) $. Taking in the account the continuity of $u$
on the remaining set $\left[ 0,\overline{t}\right] \times \Omega $, we can
say that $u$ is uniformly bounded on $\left( 0,T_{\max }\right) \times
\Omega $. Global existence becomes automatically.
\end{proof}

Since in the sense of \cite{Per-Lio}, \cite{Pie 2003} and \cite{Pie-Suz-Uma}
weak solutions are in $C(\left( 0,\infty \right) ,L^{1}(\Omega ))$, then we
can apply Theorem \ref{Max} to get the following

\begin{corollary}
If condition (\ref{Heat 1.7}) is satisfied, then positive weak solutions of (%
\ref{Heat 1.1}) are classical, global and uniformly bounded on $\left(
0,T_{\max }\right) \times \Omega $.
\end{corollary}

\begin{remark}
When the reaction satisfies the following condition%
\begin{equation*}
f(t,x,u)<0,\ \text{for all }t\in \left( t_{0},T_{\max }\right) ,\ x\in
\Omega ,\ u>0\text{,}
\end{equation*}%
then the global existence is a trivial consequence of the maximum principle.
\end{remark}

\begin{remark}
Our results are applicable to all above problems (\ref{Heat 1.2})-(\ref{Heat
1.5}) under the same assumptions on the functions $a$, $b$, $f$ , $g$, and $%
h $ and even under homogenous Dirichlet boundary conditions.
\end{remark}

\section{Applications to reaction diffusion systems}

In this section, we are concerned with the existence of globally bounded
solutions to the reaction-diffusion system 
\begin{equation}
\left\{ 
\begin{array}{c}
\\ 
\frac{\partial u}{\partial t}-a\Delta u=f(t,x,u,v), \\ 
\\ 
\frac{\partial v}{\partial t}-b\Delta v=g(t,x,u,v), \\ 
\end{array}%
\right. \quad \mathrm{in}\quad \mathbb{R}^{+}\times \Omega ,
\label{2.D.C.R.1.1}
\end{equation}%
subject to the boundary conditions%
\begin{equation}
\frac{\partial u}{\partial \eta }=\frac{\partial v}{\partial \eta }=0\qquad
on\quad \mathbb{R}^{+}\times \partial \Omega  \label{2.D.C.R.1.2}
\end{equation}%
and the initial data 
\begin{equation}
u(0,x)=u_{0}(x),\qquad v(0,x)=v_{0}(x)\qquad \mathrm{on}\quad \Omega
\label{2.D.C.R.1.3}
\end{equation}%
where $a>0$ and $b>0$ are the diffusion coefficients of some interacting
species whose spatiotemporal densities are $u$ and $v$. The domain $\Omega $
is an open bounded domain of class $C^{1}$ in $\mathbb{R}^{n}$, with
boundary $\partial \Omega $ and $\frac{\partial }{\partial \eta }$ denotes
the outward derivative on $\partial \Omega $. The initial data are assumed
to be non-negative and bounded. The reactions $f$ \ and $g$ are continuously
differentiable functions with $f$ non-negative on $\mathbb{R}^{+}\times
\Omega \times \mathbb{R}^{+2}$. Assume that%
\begin{equation}
f(t,x,u,v).g(t,x,u,v)\neq 0,\ \ \text{for all }u>0,\ v>0,\ t\in \left(
t_{0},T_{\max }\right) ,  \label{1.5mu1}
\end{equation}%
for some $t_{0}\in \left( 0,T_{\max }\right) $ and that%
\begin{equation}
f(t,x,u,v)+g(t,x,u,v)\leq C\left( u+v+1\right) ,\ \ \text{for all }u>0,\
v>0,\ t\in \left( t_{0},T_{\max }\right) .  \label{M'}
\end{equation}%
The last inequality is called "the control of mass condition".\newline
Assume that%
\begin{equation}
f(t,x,0,v),\ g(t,x,u,0)\geq 0,\ \ \text{for all}\ t\in \left( 0,T_{\max
}\right) ,\ x\in \Omega ,\text{ }u>0,\ v>0,  \label{P'''}
\end{equation}%
then using standard comparison arguments for parabolic equations (see e.g.
D. Henry \cite{Hen} and its references) the solutions remain positive at any
time.

In the case of systems on the form%
\begin{equation}
\left\{ 
\begin{array}{c}
\\ 
\frac{\partial u}{\partial t}-a\Delta u=-uF(v), \\ 
\\ 
\frac{\partial v}{\partial t}-b\Delta v=uG(v), \\ 
\end{array}%
\right. \quad \mathrm{in}\quad \mathbb{R}^{+}\times \Omega ,
\label{2.D.Exp.1}
\end{equation}%
A. Haraux and A.Youkana \cite{Har-You} generalized the results of K. Masuda
for $F(v)=G(v)=v^{\beta }$ to nonlinearities $F(s)=G(v)=e^{v^{\gamma
}},\;0<\gamma <1$. But, nothing seems to be known for instance if $%
F(v)=e^{v^{\gamma }}$ for $\gamma >1$. In the case of nonlinearities of
exponential growth like $F(v)=G(v)=e^{v}$ which appears in the
Frank-Kamenetskii approximation to Arrhenius-type reaction \cite{Ari},
Barabanova \cite{Bar} has made a small progress in this direction and proved
that solutions are globally bounded under the condition $||u_{0}||_{\infty
}\leq 8ab/(a-b)^{2}$. Then Martin and Pierre \cite{Mar-Pie} proved (in the
case $0<b<a<\infty $ which means that the absorbed substance diffuses faster
than the other one) that global solutions exist if $\Omega =\mathbb{R}^{n}$.
The proof is based on a simple comparison property concerning the kernels
associated with the operators $\left( \frac{\partial }{\partial t}-a\Delta
\right) $ and $\left( \frac{\partial }{\partial t}-b\Delta \right) .$ Some
years ago, Herrero, Lacey and Velazquez \cite{Her-Lac-Vel} obtained similar
results (in the general case). S. Kouachi and A. Youkana \cite{Kou-You}
generalized the results of A. Haraux and A. Youkana \cite{Har-You} to weak
exponential growth of the reaction term $f$. Recently in \cite{Kou(MMAS)} we
proved global existence of solutions of (\ref{2.D.Exp.1}) in a bounded
domain under the condition

\begin{equation}
\displaystyle\lim_{v\rightarrow +\infty }\frac{G^{\prime }(v)}{F(v)}=0,
\label{2.D.Exp.4}
\end{equation}%
where $G^{\prime }$ denotes the first derivative of $G$ with respect to $v$
and the functions $F$, $G$, $G^{\prime }$ and $G^{\prime \prime }$are
continuously differentiable and non-negative. \newline
To the best of our knowledge, the question of global existence of reaction
diffusion systems on a bounded domain remains open when the reactions grow
faster than a polynomial. Some partial positive results have been obtained
only when the reactions grow faster than a polynomial as it is cited above
(see \cite{Kou(MMAS)} and \cite{Bar} when $\Omega $ is bounded and \cite%
{Mar-Pie} in the unbounded case or when $\Omega =\mathbb{R}^{n}$ and their
references). The blow-up in finite time of solutions can arise in very
special cases (see for example \cite{Pie-Sch}, \cite{Pie-Sch 22} and their
references).

In this paper we show the global existence of a unique solution (uniformly
bounded on $\mathbb{R}^{+}\times \Omega $) to problem (\ref{2.D.C.R.1.1})-(%
\ref{2.D.C.R.1.3}) without conditions on the nonlinearities growth.

It is well known that, for any initial data in $L^{\infty }(\Omega )$, local
existence and uniqueness of classical solutions to the initial value problem
(\ref{2.D.C.R.1.1})-(\ref{2.D.C.R.1.3}) follows from the basic existence
theory for abstract semi-linear differential equations (see D. Henry \cite%
{Hen} or F. Rothe \cite{Rot} and their references). The solutions are
classical on $(0,T_{\max })$. Let us recall the following classical local
existence result under the above assumptions:

\begin{proposition}
\label{nu2}The system (\ref{2.D.C.R.1.1}) admits a unique classical solution 
$(u;v)$ on $(0;T_{\max }).$ If $T_{\max }<+\infty $, then%
\begin{equation*}
\underset{t\nearrow }{\lim T_{\max }}\left( \left\Vert u(t,.)\right\Vert
_{\infty }+\left\Vert v(t,.)\right\Vert _{\infty }\right) =+\infty .
\end{equation*}
\end{proposition}

Then we can formulate our main result of this section as follows:

\begin{theorem}
\label{nu1}Under conditions (\ref{1.5mu1}), (\ref{P'''}) and (\ref{M'}), the
solution of problem (\ref{2.D.C.R.1.1})-(\ref{2.D.C.R.1.3}) with positive
initial data in $\mathbb{L}^{\infty }(\Omega )$ exists globally in time and
uniformly bounded in time.
\end{theorem}

\begin{proof}
First if the two reactions $f$ and $g$ are both non positive the result is
well known and trivial: The maximum principle gives%
\begin{equation*}
\left\Vert u(t,.)\right\Vert _{\infty }\leq \left\Vert u_{0}\right\Vert
_{\infty }\text{ and }\left\Vert v(t,.)\right\Vert _{\infty }\leq \left\Vert
v_{0}\right\Vert _{\infty },\ \ \text{\ on }\left( t_{0},T_{\max }\right) ,
\end{equation*}%
and automatically we have global existence by using the continuity of the
solution on the remaining interval $\left( 0,t_{0}\right) $. Secondly if one
of the reaction is non positive (for example $f$) and the other is positive,
Then the maximum principle gives%
\begin{equation*}
\left\Vert u(t,.)\right\Vert _{\infty }\leq \left\Vert u\left(
t_{0},.\right) \right\Vert _{\infty }\text{ },\ \ \text{\ on }\left(
t_{0},T_{\max }\right) .
\end{equation*}%
To prove the boundedness of $v$ we proceed as follows: By adding the two
equations of system (\ref{2.D.C.R.1.1}), integrating over $\Omega $ and
using the inequality (\ref{M'}) with taking into the account the homogenous
boundary conditions (\ref{2.D.C.R.1.2}), we get for some positive constant $%
C $ the following%
\begin{equation*}
y^{\prime }\leq C(y+\left\vert \Omega \right\vert ),\ \ \text{\ on }\left(
0,T_{\max }\right) ,
\end{equation*}%
where%
\begin{equation*}
y(t)=\int_{\Omega }\left( u+v\right) dx,\ \ \text{\ on }\left( t_{0},T_{\max
}\right) .
\end{equation*}%
This gives%
\begin{equation*}
y(t)+\left\vert \Omega \right\vert \leq C^{\prime }e^{Ct},,\ \ \text{\ on }%
\left( t_{0},T_{\max }\right)
\end{equation*}%
and then%
\begin{equation}
\left\Vert u\right\Vert _{1}+\left\Vert v\right\Vert _{1}\leq C(T)<\infty ,\
\ \text{\ on }\left( t_{0},T_{\max }\right) .  \label{interesting inq.}
\end{equation}%
The component $v$ is $L_{1}-$uniformly bounded on $\left( t_{0},T_{\max
}\right) $, then\ from Theorem \ref{Max}\ it is uniformly bounded. Global
existence follows. The last case is when both of the reactions are positive,
then both $u$ and $v$ satisfy inequality like (\ref{interesting inq.}) via
the inequality (\ref{M'}) and another time Theorem \ref{Max} is applicable
to get global existence. This ends the proof of the Theorem.
\end{proof}

\begin{corollary}
Under conditions (\ref{1.5mu1}) and (\ref{M'}) then weak solutions of
problem (\ref{2.D.C.R.1.1})-(\ref{2.D.C.R.1.3}) with positive initial data
in $\mathbb{L}^{\infty }(\Omega )$ become classical, exist globally in time
and uniformly bounded.
\end{corollary}

\begin{remark}
The results obtained by the authors in the interesting paper \cite{Pie-Sch
22} aren't in contradiction with those in this manuscript since the
following strict Mass-control they imposed%
\begin{equation}
f(t,x,u,v)+g(t,x,u,v)\leq 0,\ \ \text{for all }u>0,\ v>0,\ x\in \Omega ,\
t\in \left( 0,T_{\max }\right) ,  \label{M}
\end{equation}%
implies automatically each of the reactions should change sign many times.
To see this: Suppose for example the reaction $f$ doesn't change sign in the
interval $\left( 0,T_{\max }\right) $, then from condition (\ref{M}) we have%
\begin{equation*}
f(t,x,u,v)\leq 0,\ \ \text{for all }u>0,\ v>0,\ x\in \Omega ,\ t\in \left(
0,T_{\max }\right) ,
\end{equation*}%
which contradicts the blow up at finite time of $u$. Then we conclude that
the two reactions change sign on $\left( 0,T_{\max }\right) \times \Omega $.
Let $\overline{t}$ the greatest $t\in \left( 0,T_{\max }\right) $ such that,
for example $f$ satisfies%
\begin{equation*}
f(\overline{t},\overline{x},\overline{u},\overline{v})=0,
\end{equation*}%
for some $\left( \overline{x},\overline{u},\overline{v}\right) \in \Omega
\times 
\mathbb{R}
_{+}^{2}$, then we have%
\begin{equation*}
\left\{ 
\begin{array}{l}
f(t,x,u,v)<0,\ \ \ \text{for all }u>0,\ v>0,\ x\in \Omega ,\ t\in \left( 
\overline{t},T_{\max }\right) , \\ 
\text{or} \\ 
f(t,x,u,v)>0,\ \ \ \text{for all }u>0,\ v>0,\ x\in \Omega ,\ t\in \left( 
\overline{t},T_{\max }\right) .%
\end{array}%
\right.
\end{equation*}%
The first alternative contradicts the blow up at finite time of $u$. The
second one gives from condition (\ref{M}) 
\begin{equation*}
g(t,x,u,v)<0,\ \ \ \text{for all }u>0,\ v>0,\ x\in \Omega ,\ t\in \left( 
\overline{t},T_{\max }\right) ,
\end{equation*}%
which contradicts the blow up at finite time of $v$. We conclude that under
the condition (\ref{M}) the two reactions possess an infinity of zeros.
\end{remark}

\section{Conclusion}

\begin{conclusion}
This article deals with the global solution for a class of parabolic
equations, specifically the semilinear heat equation under the Neumann
boundary condition and applications to some reaction diffusion systems. It
is well known that in order to demonstrate the global existence in time of
the semilinear heat equation on a bounded domain $\Omega $ of $%
\mathbb{R}
^{n}$, it is sufficient to derive a uniform bound independent of the time of
the reaction $f(t;x;u)$ on the Lebesgue space $L^{p}\left( \Omega \right) $
for some $p>n/2$. The "regularizing effect" is the name of the principle,
which ignores the sign of the reaction in the heat equation. Additionally,
the maximal principle reveals the global existence of the solution when the
reaction is nonpositive. To our knowledge, there is no information on the
global existence of the solution when the reaction is positive on the
interval~of the local existence, unless some partial results occur under
conditions such as$\underset{a}{\overset{+\infty }{\int }}\frac{ds}{f\left(
s\right) }=+\infty $ \ for some positive constant $a$ with the restriction
reaction $f\left( s\right) $ doesn't depend $\left( t,x\right) $. In this
manuscript we show that if the reaction $f(t,x,u)$\ is strictly positive,
then weak solutions (i.e. solutions belonging to the Lebesgue space $%
L^{1}\left( \left( 0,T_{\max }\right) \times \Omega \right) $) become global
classical solutions. That is to prove global existence it suffices in
addition to the positivity of the reaction to suppose its uniform
boundedness on the Lebesgue space $L^{1}\left( \left( t_{0},T_{\max }\right)
\times \Omega \right) $ for some positive $t_{0}\in \left( 0,T_{\max
}\right) $. Then we present some applications to a class of reaction
diffusion system and prove the global existence of theirs positive weak
solutions under the unique condition%
\begin{equation*}
f(t,x,u,v).g(t,x,u,v)\neq 0,\ \ \text{for all }u>0,\ v>0,\ t\in \left(
t_{0},T\right) ,
\end{equation*}%
where $f$ and $g$ denote the reactions of the system.\ The proof is based on
the maximum principle.
\end{conclusion}

\end{document}